\documentclass[a4paper, 11pt]{amsart}
\input xy   
\xyoption{all}
\swapnumbers
\usepackage{amssymb} 
\usepackage{xcolor} 
\usepackage{amsthm}
\usepackage{enumerate}
\usepackage[active]{srcltx}
\usepackage{colordvi}

\numberwithin{equation}{section}

\theoremstyle{plain}
  \begingroup
        \newtheorem{theorem}[equation]{Theorem}
        \newtheorem{lemma}[equation]{Lemma}
        \newtheorem{proposition}[equation]{Proposition}
        \newtheorem{corollary}[equation]{Corollary}

	    \newtheorem{definition}[equation]{Definition}

\endgroup

\theoremstyle{definition}
  \begingroup
        \newtheorem{notation}[equation]{Notation}

        \newtheorem*{observation}{Observation}

\endgroup

\newcommand{\mr}[1]{\buildrel {#1} \over \longrightarrow}
\newcommand{\ml}[1]{\buildrel {#1} \over \longleftarrow}
\newcommand{\Mr}[1]{\buildrel {#1} \over \Longrightarrow}
\newcommand{\Ml}[1]{\buildrel {#1} \over \Longleftarrow}

\newcommand{\Eop}[1]{\cc{E}ns^{\cc{#1}^{op}}}
\newcommand{\E}[1]{\cc{E}ns^{\cc{#1}}}

\newcommand{\pmr}[2]
{
\xymatrix@C=5ex@R=2.4ex
         {
          {} \ar@<1.6ex>[r]^{#1} 
	         \ar@<-1.1ex>[r]^{#2} & {}
         }
}

\newcommand{\pml}[2]
{
\xymatrix@C=5ex@R=2.4ex
         {
            {} 
          & {} \ar@<1.0ex>[l]_{#1} 
	           \ar@<-1.7ex>[l]_{#2}
         }
}

\newcommand{\pmrl}[2]
{
\xymatrix@C=5ex@R=2.4ex
         {
            {} \ar@<1.6ex>[r]^{#1} 
          & {} \ar@<1ex>[l]_{#2}
         }
}

\newcommand{\pmdu}[2]
{
\xymatrix@C=2.4ex@R=5ex
         {
            {} \ar@<1.6ex>[d]^{#1} 
         \\ {} \ar@<1ex>[u]_{#2}
         }
}

\newcommand{\cellr}[3]
{
\xymatrix@C=7ex@R=2.4ex
         {
          {} \ar@<1.6ex>[r]^{#1} 
          \ar@{}@<-1.3ex>[r]^{\!\! #2 \, \!\Downarrow}
                                         \ar@<-1.1ex>[r]_{#3} & {}
         }
}

\newcommand{\celll}[3]
{
\xymatrix@C=7ex@R=2.4ex
         {
            {} 
          & {} \ar@<1.0ex>[l]^{#1} 
          \ar@{}@<-1.7ex>[l]^{\!\! #2 \, \!\Downarrow}
	                                 \ar@<-1.7ex>[l]_{#3}
         }
}

\newcommand{\limite}[2]{\underrightarrow{lim}_{\; #1} {\; #2}}
\newcommand{\colimite}[2]{\underrightarrow{colim}_{\; #1} {\; #2}}

\newcommand{\Colim}[1]{\underrightarrow{\cc{L}im}{\; #1}}
\newcommand{\Lim}[1]{\underleftarrow{\cc{L}im}{\; #1}}

\newcommand{\cc}{\mathcal}
\newcommand{\bb}{\mathbb}
\newcommand{\ff}{\mathsf}
\newcommand{\nn}{\mathbf}
\newcommand{\rr}{\mathrm}

\newcommand{\lm}{\ell}
\newcommand{\im}{\imath}
\newcommand{\jm}{\jmath}

\newcommand{\mmr}[1]{\buildrel {#1} \over \hookrightarrow}

\newcommand{\cqd}{\hfill$\Box$}

\newcommand{\ain}{\;\widetilde{\in}\;}

\begin{document}

\title{From Yoneda to Topoi morphisms}

\author{Eduardo J. Dubuc}

%\date{\vspace{-5ex}}

\begin{abstract}
 In this note we show how two fundamental results in Topos theory follow by repeated use of Yoneda's Lemma, the formalism of natural transformations and very basic category theory. 
 
 In Lemma \ref{yoneda6}, we show the fundamental result SGA4 EXPOSE IV Proposition 4.9.4, which says that for any site $\cc{C}$, the canonical functor $\cc{C} \mr{\varepsilon} Sh(\cc{C})$ into the category of sheaves, classifies sites morphisms $\cc{C} \mr{} \cc{Z}$ into any topos $\cc{Z}$. After the usual Yoneda's Lemma, Lemma \ref{yoneda1}, we show that the  Yoneda functor $\cc{C} \mr{h} \Eop{C}$ classifies functors 
$\cc{C} \mr{} \cc{Z}$ into any cocomplete category $\cc{Z}$, via a cocontinuos  extension $\Eop{C} \mr{} \cc{Z}$, Lemma \ref{yoneda3}. Then
we reach Lemma \ref{yoneda6} by an step by step enrichment of \ref{yoneda3}. All we use is Yoneda's Lemma, over and over again, and the Yoga of natural transformations. In Lemma \ref{yoneda7} we show the equivalence between \emph{flatness}  and  \emph{left exactness} for functors from finitely complete categories into any topos. Our proof is elementary, we show how basic exactness \mbox{properties} of sets prove the result for set valued functors, then we generalize to functors valued in any topos utilizing results of the previous sections and the Yoga of natural transformations. Besides the thread we \mbox{follow,} nothing here is new, although we haven't seen \ref{yoneda7} proved this way before.

\end{abstract}

\maketitle

\section{Introduction}  \label{intro}
%\noindent {\bf Introduction.}
This paper gained its impetus from a talk amongst Matias del Hoyo and the author, we are very grateful to Matias' input.
He suggested that the three Yoneda's lemmas discussed  in my category theory course concerning the Yoneda functor 
$\cc{C} \mr{h} \Eop{C}$:

\vspace{1ex}  

Yoneda I: Yoneda's lemma, 

Yoneda II: Density of the functor $h$, 

Yoneda III: The functor $h$ classifies functors 
$\cc{C} \mr{p} \cc{Z}$ into any cocomplete category 
$\cc{Z}$, via a cocontinuos  extension 
$\Eop{C} \mr{\widetilde{p}} \cc{Z}$.

\vspace{1ex}

 can be followed by another three:  

\vspace{1ex}

Yoneda IV: The functor $h$ in Yoneda III classifies functors 
$\cc{C} \mr{p} \cc{Z}$ into any cocomplete category 
$\cc{Z}$, via a cocontinuos  extension $\widetilde{p}\,$  furnished with an explicit right adjoint
$\cc{Z} \mr{h_p} \Eop{C}$.
 
Yoneda V: The functor 
$h$ in Yoneda III classifies \emph{flat} functors 
\mbox{$\cc{C} \mr{p} \cc{Z}$} into any cocomplete category 
$\cc{Z}$, via a \emph{left exact} cocontinuos extension
$\widetilde{p}$.
 
%$\Eop{C} \mr{\widetilde{p}} \cc{Z}$.  
%$p$ is flat, $\widetilde{p}$ is left exact.  

Yoneda VI: If $\cc{C}$ is a site, the canonical functor 
$\cc{C} \mr{\varepsilon} Sh(\cc{C})$,
$\varepsilon = a\,h$ (where "a" denotes the associate sheaf functor), classifies site morphisms 
\mbox{$\cc{C} \mr{p} \cc{Z}$} 
into any topos $\cc{Z}$, 
via the restriction of the cocontinous extension 
$\widetilde{p}$ 
to the subcategory of sheaves $Sh(\cc{C}) \subset \Eop{C}$.

\vspace{1ex}

Finally we add as a seventh Yoneda the important fact that in the presence of finite limits on the site, left exactness is a sufficient condition for flatness. 

\vspace{1ex}

Yoneda VII: If a category $\cc{C}$ has finite limits, a functor $\cc{C} \mr{F} \cc{Z}$ into a topos $\cc{Z}$ is flat provided it is exact.

\vspace{1ex}
 
In this note we follow this thread, adding the full explicit  2-categorical aspects of the statements, and complete detailed proofs of them. 
In this way we exhibit how each one follow from the previous ones and just the formal properties (Yoga) of natural transformations. Even to pass from the purely category theory Yoneda V to the topos theory Yoneda VI, we only need the pertinent definitions and an almost tautological arguing. For Yoneda VII we also need a very basic category theory proof that the statement holds for set valued functors, then we show how its generalization to functors valued in any topos follows from the previous Yonedas, we haven't seen this proof before in the literature, as far as we know it is inedit.

\section{Preliminaries on Categories}  \label{cat_prelim}
%{\bf Preliminaries on Categories.} \label{cat_prelim}
We recall some category theory definitions that we shall use in order to have explicit and rigorous proofs, and also in this way fix notation 
and terminology. We refer to Grothendieck's theory of universes \cite{G1}, see also \cite{KS}. Let $S$ and $T$ be any two sets, the notation $S \ain T$ means that $S$ is bijective with an element of $T$. 

\vspace{1ex}

Let $\cc{C}$ be a category, $X,\; Y$ any two objects, we denote 
$[X,\; Y]$ the homset of morphisms between $X$ and $Y$. 
Respect to a universe $\cc{U}$ we use the following dictionary: 
$$
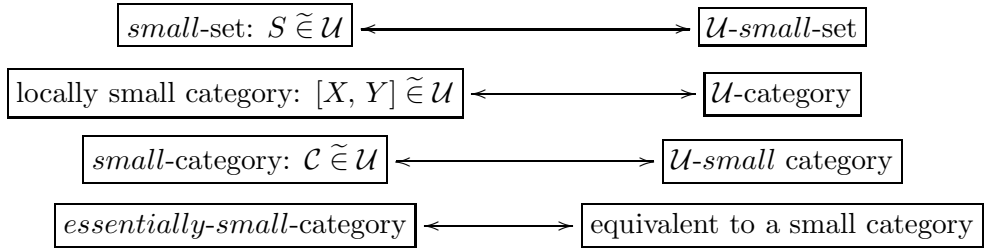
\xymatrix@R=0.1ex@C=8ex
    {
     \fbox{$small$-set: $S \ain \cc{U}$}  \ar@{<->}[r]  
   & \fbox{$\cc{U}$-$small$-set}
  \\
     \fbox{locally small category: $[X,\,Y] \ain \cc{U}$}  \ar@{<->}[r]  
   & \fbox{$\cc{U}$-category}
  \\
     \fbox{$small$-category: $\cc{C} \ain \cc{U}$}  \ar@{<->}[r]  
   & \fbox{$\cc{U}$-$small$ category}
  \\
     \fbox{$essentially$-$small$-category}  
                                                     \ar@{<->}[r]  
   & \fbox{equivalent to a small category}
    }
$$
We denote by $\cc{E}ns$ the category of small sets. Given a category 
$\cc{C}$, we denote $\Eop{C}$ the  category of presheaves over $\cc{C}$. Its objects are the functors 
\mbox{$F: \Eop{C} \mr{} \cc{E}ns$,} and its arrows are the natural transformations $\theta:F\mr{} G$. Given a locally small category, the \emph{covariant} Yoneda functor 
$\cc{C} \mr{h} \Eop{C}$ is defined by \mbox{$h_X = [-,\, X]$,} and for $X \mr{f} Y$, $h_f = f_\ast$ $=$ post-composition with $f$. We will often abuse the notation and omit to indicate the label "$h$". When we say that an object $X$ is considered in $\Eop{C}$, we mean $h(X)$ (see \ref{compatibilidad} below).
\vspace{1ex}

\vspace{1ex}

{\bf Convention.} \label{convention}
To improve readability and simplify the language we adopt the following convention: A \emph{category} means a locally small category, a locally small category which is not \emph{small} is called \emph{large}, and categories which are not even locally small are called \emph{illegitimate}. 

\vspace{1ex}

\emph{If $\cc{C}$ is a small category, $\Eop{C}$ is a 
large category, if $\cc{C}$ is a large category, $\Eop{C}$ is an 
illegitimate category} 

\begin{definition} \label{sef}
A family $X_i \mr{\lambda_i} X$ is \emph{strict epimorphic} if for every compatible family $X_i \mr{f_i} Y$, i.e.    
$\forall x, y  \;(\lambda_i x = \lambda_i y \; \Rightarrow f_i x = f_i y)$,
there exists a unique $X \mr{f} Y$ such that $f\,\lambda_i = f_i$, as indicated in the following diagram:
$$ 
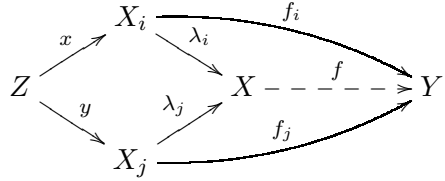
\xymatrix@R=2.5ex
        {
         & X_i \ar[dr]^{\lambda_i}
               \ar@/^2ex/[drrr]^{f_i}
        \\
           Z  \ar[ur]^{x}
              \ar[dr]^{y} 
        && X \ar@{-->}[rr]^{f} 
        && Y 
       \\
         & X_j \ar[ur]^{\lambda_j}
               \ar@/_2ex/[urrr]^{f_j}
         }
$$
It is said \emph{universal} if under any change of base it remains strict epimorphic.
\end{definition}

\section{Yoneda I: Yoneda's Lemma}  \label{sec_yoneda1}
%
%{\bf \emph{Yoneda I}:} \emph{Yoneda's Lemma}. 
%
Let $\cc{C}$ be a category: 
\begin{lemma}[Yoneda I] \label{yoneda1}
For any $C \in \cc{C}$ and any functor $\cc{C}^{op} \mr{F} \cc{E}ns$, there is a bijective correspondence $\eta$, natural in the variables $C$ and $F$:
$$
\xymatrix@R=0.1ex@C=0.5ex
   {
     {} & [-,\, C] \ar[rr]^\theta & {}  & F
    \\
     {} \ar@{-}[rrrrr] &&&&& \eta
    \\
     {} & {} & \xi \in FC
   }
\hspace{4ex}
\xymatrix@R=0.1ex 
   {
     {} 
    \\
     Nat([-,\, C],\, F) \mr{\eta} FC 
    \\
     {} 
   }
$$
To a natural transformation $\theta$, we define 
$\eta(\theta) = \theta_C(id_C) \in FC$. Then, 
$$
\forall \, \xi \, \exists \, ! \;\; \theta \;\; | \;\; \xi = \theta_C(id_C). %\; i.e. \;\; | \;\; \xi = \eta(\theta).
$$

\end{lemma}

%Given an element $\xi \in F(C)$, 
\begin{proof} Given $f \in [X,\,C]$, it is straightforward to check that the condition of naturality on $\theta$ forces the definition \mbox{$\theta_X(f) = F(f)(\xi)$}. The rest of the proof consists also of straightforward checks.   
\end{proof}
\begin{corollary}
The Yoneda functor $\cc{C} \mr{h} \Eop{C}$ is fully faithful. \cqd
\end{corollary}
The following \emph{abuse of notation} introduced by Grothendieck 
\mbox{(SGA4, I, 141)} is at the core of Category Theory.
\begin{notation} \label{compatibilidad}
We identify $[-,\, X]$ with $X$ and $\theta$ with $\xi$, denoting both with a same label. In this way we have:
$$
x \in FX  \; \equiv \; X \mr{x} F.
$$
Given $X \mr{f} Y$, $F \mr{\varphi} G$, we also denote $f$ the natural transformation $f_\ast$. The naturality in the object and the functor variables means: 
$$
For \; y \in FY: \;\; F{f}(y) \in FX  \; \equiv \;  
X \mr{f} Y \mr{y} F.
$$
$$   
For \; x \in FX: \;\; \varphi_X(x) \in GX \; \equiv \; 
X \mr{x} F \mr{\varphi} G. 
$$
That is, $F$ and $\varphi$ act by composing, which legitimate this abuse of notation.
\end{notation}

\section{Yoneda II: Any presheaf is colimit of representable ones} \label{sec_yoneda2}
%
%{\bf \emph{Yoneda II}:} \emph{Any presheaf is colimit of representable ones}.
%
Given a category $\cc{C}$ and a functor $\cc{C} \mr{F} \cc{E}ns$, the \emph{Diagram} or \emph{Category of elements} of $F$, $\Gamma_F \mr{} \cc{C}$, has as objects pairs 
$(x,\,X),\, x \in FX$, and arrows defined by:
$$
\xymatrix@R=0.1ex@C=0.3ex
   {
    && (x, \, X) \ar[r]^f & (y,\,Y)
    \\
     \ar@{-}[rrrrrr] &&&&&&& {}
    \\
    &&  {X \mr{f} Y \;\; | \;\; F(f)(y) = x} & {\equiv} & { x = y \circ f}
   }
$$
There is a functor $\Diamond :\; \Gamma_F \mr{} \Eop{C}$, 
$\Diamond(x,\, X) = [-,\, X] (= X)$. The condition in the definition of arrows of $\Gamma_F$ means that the arrows 
$X \mr{\lambda_{(x,\,X)}} F$,  
$\lambda_{(x,\,X)} = x$, determine a cone 
$\Diamond \mr{\lambda} F$ of 
$\Diamond$ over $\Gamma_F$. 
\begin{lemma}[Yoneda II] \label{yoneda2} 
For any \mbox{$F \in \Eop{C}$, $H \in \Eop{C}$,} there is  a bijective correspondence: 
$$
\xymatrix@R=0.1ex@C=0.3ex
   {
    && F \mr{\varphi} H : & \{FX \mr{\varphi_X} HX\},
    &  natural \; transformations 
    \\
     \ar@{-}[rrrrrr] &&&&&&& {}
    \\
    && \Diamond \mr{\mu} H : & \{X \mr{\mu_{(x,\,X)}} H\},
    &  cones \; of \;  \Diamond \; over \; \Gamma_F
   } 
$$
To a natural transformation $\varphi$ we assign the cone defined by the composite 
$\Diamond  \mr{\lambda} F \mr{\varphi} H$. Then,
$$
\forall \; \mu  \;\; \exists \, ! \;\; \varphi \;\;\; | \;\;\;
\mu = \varphi \circ  \lambda
$$ 
%  
%\varphi_{(x,\, X)}
%
That is, $(\Diamond  \mr{\lambda} F)$ $=$  
\{$[-,\, X] (= X) \mr{\lambda_{(x,\,X)}} F$\} is a colimit cone, we write:
$$
F = \colimite{\Gamma_F}{\Diamond} \;\;=\;\; 
\colimite{(x,\,X)}{h_X} \;\;=\;\; 
\colimite{(x,\,X)}{[-,\,X]}.
$$
\begin{proof} 
Given a cone $\mu$ and given 
$(X \mr{x} F)$, the only possible definition of 
$(X \mr{\varphi_X(x)} H)$ such that $\mu = \varphi \circ  \lambda$ is given by  
  \mbox{$\varphi_X(x) (\equiv \varphi \circ x) = \mu_{(x,\, X)}$.} Then, given $X \mr{f} Y$, it is straightforward to check that the naturality condition for $\varphi$ holds by the cone condition of $\mu$. 
\end{proof} 
\end{lemma}
\section{Yoneda III: The co-continuous extension of a functor} \label{sec_yoneda3}
%
%{\bf \emph{Yoneda III}:} \emph{The co-continuous extension of a functor}
%
For any category $\cc{C}$ and any functor 
$\cc{C} \mr{p} \cc{E}ns^{op}$ we have the following diagram:
$$
\xymatrix@C=4ex@R=3ex
   { 
    {\cc{C}}  \ar[rr]^{h}
          \ar@/_2ex/[rrdd]_{p}
   & {}   \ar@{}[dd]^(.35){\approx\,\eta}
   & {\Eop{C}}  \ar[dd]^{\widehat{p}}
   \\
   \\
   & {}
   & {\cc{E}ns^{op}}
   }
$$
where $\widehat{p} = Nat(-,\, p)$, and 
$\eta:\,p \mr{\approx} \widehat{p}\, h$ is defined by, for $X \in \cc{C}$, 
%\mbox{$FC \ar[r]^(.3){\eta}_(.3){\approx} & Nat([-,\, C],\, p)$}
\mbox{$FC \mr{\eta_C} Nat([-,\, C],\, p)$.} Note that this assignation is functorial in the variable $p$.

Clearly $\widehat{p}$ is cocontinous, and the formal properties of natural transformations yield that the functors $\widehat{(-)}$ and $h^\ast$  actually determine an equivalence of categories: 
$$
[\cc{C},\, \cc{E}ns^{op}] \pmrl{\widehat{(-)}}{h^\ast}  
                                      Cont[\Eop{C},\, \cc{E}ns^{op}].
$$ 
In fact, given $p$, we already have a natural isomorphism $p \mr{\approx} h^\ast \widehat{p}$ \mbox{(Yoneda I),} and given $\Eop{C} \mr{G} \cc{E}ns$, a natural isomorphism 
$\widehat{h^\ast G} \mr{\approx} G$ follows since 
$\widehat{h^\ast G}\,h = \widehat{G h}\,h \approx G h$ and both functors are cocontinous \mbox{(use Yoneda II).} 

We now pass to generalize this statement for any cocomplete  category $\cc{Z}$ in place of $\cc{E}ns^{op}$. 
We have the following commutative diagram whose construction  we now explain step by step:
$$ 
\xymatrix@C=4.5ex@R=4ex
   {
    {\cc{C}}  \ar[rr]^{h}  
              \ar[rdd]_{p} 
              \ar@/_8ex/[ddddrr]_{p_Z} 
              \ar@{}[drr]_{\widetilde{\eta}}
   && {\Eop{C}} \ar[ldd]^{\hspace{-1ex}\widetilde{p}} 
                \ar[rdd]^{\widehat{p}} 
                \ar@/^21ex/[dddd]^{\widehat{p_Z}}
                \ar@{}[dd]^{\eta'} 
   \\
   && {}
   \\ 
   & {\cc{Z}} \ar[rr]^{l} 
              \ar[rdd]_{[-,\,Z]} 
   & {}       \ar@{}[dd]^{\equiv}
   & {(\E{Z})^{op}} \ar[ldd]^{ev_Z}
   \\
   \\
   && {\cc{E}ns^{op}}
   }
$$
where $l$ is the contravariant Yoneda functor and $ev_Z$ is evaluation at $Z$. 

\vspace{1ex}

{\bf 1.} Clearly $[-, \,Z] = ev_Z \, l$. 

\vspace{1ex}

{\bf 2.} Set $p_Z = ev_Z \,l\,p = [p(-),\, Z]$. Note that $p_Z$ is also functorial in the variables $p$ and $Z$. 
We have $\widehat{p_Z}$ and  
$p_Z \mr{\eta_Z} \widehat{p_Z}\,h$. 

\vspace{1ex}

{\bf 3.}
We have $\widehat{p}$ such that $ev_Z \,\widehat{p} = p_Z$\,: Define $\widehat{p}$ by, for 
$H \in \Eop{C}$, $\widehat{p}(H)(Z) = \widehat{p_Z}(H) =  Nat(H,\, p_Z)$, clearly $\widehat{p}$ has all the functoriality requirements.
 Since $\widehat{p_Z}$ is cocontinuos and colimits are computed pointwise, it follows that $\widehat{p}$ is cocontinuos. 
 
 \vspace{1ex}  
 
{\bf 4.} The family of all $\eta_Z$ determines  
$l\,p \mr{\eta''} \widehat{p} \,h$ such that 
$ev_Z \, \eta'' =  \eta_Z$, 
\mbox{$(ev_Z \,l\,p \mr{ev_Z \, \eta''} ev_Z\, \widehat{p} \,h)$} $=$  
$(p_Z \mr{\eta_Z} \widehat{p_Z} \, h)$. 
 
\vspace{1ex}

{\bf 5.} Given $H \in \Eop{C}$, define 
$\widetilde{p}{H} \in \cc{Z}$ as a colimit cone in $\cc{Z}$, 
\mbox{$\{pX \mr{x} \widetilde{p}{H}\}_{(x,\, X) \in \Gamma_H}$}. Since $\widehat{p}$ as well as $l$ are cocontinuos it follows we have
$\,l\,\widetilde{p}\, H  \mr{\eta'_H} \widehat{p}\,H$ such 
that 
%\end{document}
$\xymatrix 
    {
     \{l_{pX}  \ar[r]^{\eta''_X} 
             \ar[d]^{x} 
   & \widehat{p}\,h_X \ar[d]^{x}\}
  \\              
     l\,\widetilde{p}\, H  \ar[r]^{\eta'_H}
   & \widehat{p}\,H 
    }
$
$\hspace{-1.5ex} _{(x,\, X) \in \Gamma_H}$. Since $l$ is full and faithful (Yoneda I), it follows 
$\widetilde{p}$ is a functor in such a way that 
$l\, \widetilde{p} \mr{\eta'} \widehat{p}$ becomes a natural transformation.

\vspace{1ex}

{\bf 6.} Let $l\,p \mr{\theta} l\,\widetilde{p}\,h\,$ be 
$\;\theta = 
(l\,p \mr{\eta''} \widehat{p}\,h \mr{\eta'^{-1}\,h} 
l\,\widetilde{p}\,h)$.  
Since $l$ is full and faithful (YI) it follows there is 
$p \mr{\widetilde{\eta}} \widetilde{p}\,h$ such that 
$\eta'' = \eta'\,h \circ l\,\widetilde{\eta}$.

Relabeling  $\widetilde{\eta}$, again  as $\eta$ we have:
\begin{proposition} \label{preyoneda3}
For any cocomplete category $\cc{Z}$ and
any functor \mbox{$\cc{C} \mr{p} \cc{Z}$,} there exists a cocontinuos functor $\Eop{C} \mr{\widetilde{p}} \cc{Z}$ together with a natural isomorphism 
$p \mr{\eta} \widetilde{p}\,h$: 

$$
\xymatrix@C=4ex@R=3ex
   { 
    {\cc{C}}  \ar[rr]^{h}
          \ar@/_2ex/[rrdd]_{p}
   & {}   \ar@{}[dd]^(.35){\approx\,\eta}
   & {\Eop{C}}  \ar[dd]^{\widetilde{p}}
   \\
   \\
   & {}
   & {\cc{Z}}
   }
   \hspace{4ex} 
%\footnote
%{
%By definition of Kan Extension this proposition says that  
%$\widetilde{(-)}$ is in particular the left Kan extension of $p$ along $h$;  
%$\widehat{p} = Lan_h(p).$
%}
%\cqd
$$

\end{proposition}
\begin{lemma}[Yoneda III] \label{yoneda3}
For any cocomplete category $\cc{Z}$, 
precomposition with $h$, $h^\ast$, 
establishes an equivalence of categories, with a quasi-inverse given by the cocontinuos extension $\widetilde{(-)}$:
 
$$
[\cc{C},\, \cc{Z}] \pmrl{\widetilde{(-)}}{h^\ast}  
                                 Cont[\Eop{C},\, \cc{Z}].
$$
\end{lemma}
\begin{proof} 
In fact, given $p$, we already have 
$p \mr{\approx} h^\ast \widetilde{p}$, and given $\Eop{C} \mr{G} \cc{Z}$, a natural isomorphism 
$\widetilde{h^\ast G} \mr{\approx} G$ follows since 
\mbox{$\widetilde{h^\ast G}\,h = \widetilde{G h}\,h \approx G h$} and both functors are cocontinous (use Yoneda II).
\end{proof} 
\section{Yoneda IV: Right adjoint for the cocontinuos extension} \label{sec_yoneda4} 
%
%{\bf \emph{Yoneda IV}:} \emph{Right adjoint for the co-continuos %extension}
%
Given any two categories $\cc{X},\; \cc{Y}$, we define the category $LA[\cc{X},\, \cc{Y}]$ of left adjoint functors and natural transformations,  more precisely: 

\begin{definition}[LA] \label{leftadjoint}
The category $LA[\cc{X},\, \cc{Y}]$: its objects are triples 
\mbox{$u = (u^*, u_*, \varphi)$,}  
with functors $\cc{X} \pmrl{u^*}{u_*} \cc{Y}$, and an adjunction $\varphi$:  
$u^* \dashv u_*$. Given two objects $u,\, v$, an arrow is a natural transformation $u^* \mr{\alpha} v^*$ (note that by means of the adjunctions,  $\alpha$ uniquely determines a natural transformation 
$v_* \mr{\beta} u_*$ and vice-versa). Composition is just the vertical composition of natural transformations
\end{definition}
We will see that the cocontinuos extension determines the left adjoint component $u^*$ of an object in this category: 

\vspace{1ex}

For any functor $\cc{C} \mr{p} \cc{Z}$ we enlarge the diagram in section 
\ref{sec_yoneda3} with a functor $h_p$ as follows: 
$$
\xymatrix@C=4ex@R=3ex
   { 
    {\cc{C}}  \ar[rr]^{h}
          \ar@/_2ex/[rrdd]_{p}
   & {}   \ar@{}[dd]^(.35){\approx\,\eta}
   & {\Eop{C}}  \ar@{<-}@<1.5ex>[dd]^{h_p} 
                \ar@<-1.2ex>[dd]^{\!\widetilde{p}}
   \\
   \\
   & {}
   & {\cc{Z}}
   }
$$
\begin{definition} \label{def_hp}
$h_p$ is defined as, for $Z \in \cc{Z}$, 
$h_p(Z) = p_Z = [p(-),\, Z]$, clearly functorial in the variable $Z$.
\end{definition}
%
% Note that $p_Z$ is also functorial in the variables $p$ and $Z$. 
%We have $\widehat{p_Z}$ and  
%$p_Z \mr{\eta_Z} \widehat{p_Z}\,h$. 
%
Given a functor $R \in \Eop{C}$, the colimit cone over its category of elements $\Gamma_R$, 
\{$h_X  \mr{\lambda_{(x,\,X)}} R\}_{(x,\,X) \in \Gamma_R}$, $\lambda_{(x,\,X)} = x$ (Yoneda II), determines the following diagrams, with arrows 
$\alpha$ as indicated:
$$
\xymatrix@C=7ex
   {
    h_X \ar[r]_{\lambda_{(x,\, X)}} 
        \ar@/^4ex/[rr]^{\alpha_{(x,\, X)}}
  & R \ar[r]_{\alpha}
  & h_p(Z)
   }    
\hspace{3ex}
\xymatrix@C=7ex
   {
    pX \ar@/^4ex/[rrr]^{\alpha_{(x,\, X)}}
       \ar[r]_{\approx\,\eta}
  & \widetilde{p}(h_X)  \ar[r]_{\widetilde{p}(\lambda_{(x,\, X)})} 
  & \widetilde{p}(R) \ar[r]_{\alpha}
  & Z
   }
$$
Note that the cones with vertices $R$ and $\widetilde{p}(R)$ are both colimit cones. Given $\alpha$ in the left diagram, composing with 
$\lambda_{(x,\, X)}$ determines bijectively  a cone of natural transformations $\alpha_{(x,\, X)}$. By Yoneda's Lemma (Yoneda I) this corresponds bijectively to  
\mbox{$\alpha_{(x,\, X)} \in h_p(Z)(X)$,} 
which by definition of 
$h_p$, is a cone $pX \mr{\alpha_{(x,\, X)}} Z$ in the right diagram, which in turn determines bijectively an arrow 
$\alpha$ in the second diagram. This establishes an adjunction  $\varphi$ between the functors 
$\widetilde{p}$ and $h_p$.

%a cone 
%$\alpha_{(x,\, X)}$ in the first diagram is the same thing that a cone
%$\beta_{(x,\, X)}$ in the second diagram. Since the cones with vertices $R$ and $\widetilde{p}(R)$ are both colimit cones, it follows a bijection between arrows $\alpha$ and $\beta$, showing an adjunction between the functors $\widetilde{p}$ and $h_p$.  

From Yoneda III and some checking it follows: 
\begin{lemma}[Yoneda IV] \label{yoneda4}
For any cocomplete category $\cc{Z}$, 
precomposition with $h$, $h^\ast$, 
establishes an equivalence of categories, with a quasi-inverse given by the cocontinuos extension
 $\widetilde{(-)}$:
$$
[\cc{C},\, \cc{Z}] \pmrl{\ell}{h^\ast} 
 LA[\Eop{C},\, \cc{Z}]
$$
where $\ell(p) = (\widetilde{p},\,h_p,\, \varphi)$ as defined in \ref{def_hp}, and $h^*(u^*,\, u_*,\, \varphi) = u^* h.$ 
\end{lemma}
%
%
%{
%\color{blue}  
%Comentario: Hay que poner una definicion de la categoria 
%$LA[\cc{X},\, \cc{Y}]$ al estilo de la definicion de "Morphismes de %Topos" en el SGA4. 
%} 
%
%
\section{Yoneda V: Flatness} \label{sec_yoneda5} 
%{\bf \emph{Yoneda V}:} \emph{Flatness}
%
This section is tautological. By definition a functor 
$\cc{C} \mr{p} \cc{Z}$ is \emph{flat} if its cocontinous extension 
$\Eop{C} \mr{\widetilde{p}} \cc{Z}$ is \emph{exact}, i.e., it  preserves finite limits. Then, as full subcategories in Yoneda IV we have: 
\begin{lemma}[Yoneda V] \label{yoneda5}
The statement in \ref{yoneda4} holds for the following diagram of full subcategories:
$$
Flat[\cc{C},\, \cc{Z}] \pmrl{\ell}{h^\ast}  ELA[\Eop{C},\, \cc{Z}],
$$ 
where the objects of $Flat$ are the flat functors, and those of $ELA$  are the ones of $LA$ with exact $u^*$. 
\end{lemma} 
\section{Preliminaries on topos theory} \label{top_prelim}  

We recall some basic topos theory and definitions.

\begin{definition}[SGA4 II, Def. 1.3 and Def. 3.0.2] \label{site_def}
A \emph{site} is a category $\cc{C}$ together with a \emph{pretopology}, that is, for each object $X \in \cc{C}$, a collection $Cov(X)$ of families
$X_i \mr{} X $ subject to the three usual axioms, and the following assumption, \emph{Topological generators}: There exists a small set of objects  $\cc{G}$ which covers every object, that is, such that for every $X$ there is a cover $X_i \mr{} X  \in  Cov(X)$ with all the $X_i \in \cc{G}$.

The \emph{canonical} pretopology has as covers all universal strict epimorphic families (see Definition \ref{sef}). A site is said to be \emph{subcanonical} when the covers are strict epimorphic families. 
\end{definition}
%
%{\color{cyan}  Definicion de pretopologia SGA4 II Def. 1.3, definicion %de sitio SGA4 II Def. 3.0.2}
%
%Recall the definition of sheaf:
%
\begin{definition} \label{sheaf_def}
A presheaf  $H \in \Eop{C}$ is a sheaf if it believes that the covers $X_i \mr{} X$ are strict epimorphic families in 
$\Eop{C}$ \mbox{(see Definition \ref{sef}).}
\end{definition}

A site is subcanonical if and only if the representable functors are sheaves.

\begin{corollary} [SGA4 II Corollaire 4.11] \label{ls_Sh}
Natural transformations between sheaves are determined by their values on the topological generators, thus sheaves determine a full subcategory $Sh(\cc{C}) \subset \Eop{C}$
of the illegitimate category of presheaves.
\end{corollary}
Note that unless $\cc{C}$ is a small category, the illegitimate category of presheaves  $\Eop{C}$ is not a category, in particular it is not a topos (see definition \ref{topos_def} below). 

\begin{theorem} [SGA4 II 4.4.0] \label{sheaf_ass}
\emph{The associated sheaf functor}: There is a left exact left adjoint functor (necessarily cocontinuos),  $a: \Eop{C} \mr{} Sh(\cc{C})$ to the inclusion $ i: Sh(\cc{C}) \subset \Eop{C}$, $a \dashv i$,\; 
$id \approx a\,i$,\; $id \Rightarrow i\,a$.
\end{theorem}

\begin{definition} \label{canonico}
The canonical functor $\varepsilon: \cc{C} \mr{} Sh(\cc{C})$ is defined to be the composite $\varepsilon = a\,h$ of the associated sheaf with the yoneda functor. 
\end{definition}

By definition the canonical functor is exact and sends covering families into strict epimorphic families. When the site is subcanonical it is fully faithful and coincides with the yoneda functor, 
$i \, \varepsilon \approx h$. 

\begin{definition} [SGA4 IV Definition 1.1] \label{topos_def}
A \emph{Topos} is a category $\cc{Z}$ such that there exists a site
$\cc{C}$ such that $\cc{Z}$ is equivalent to the category 
$Sh(\cc{C})$ of sheaves on $\cc{C}$ (note that the presheaf (illegitimate) category $\Eop{C}$ is not a topos unless the site $\cc{C}$ is small). 
\end{definition}
%
%Since the associated sheaf functor is cocontinous  \ref{sheaf_ass}, from \ref{yoneda2} it immediately follows:
%
%\begin{lemma}[density of the  canonical functor] 
%\label{eyoneda2}
%For any site $\cc{C}$ and sheaf $p \in Sh(\cc{C})$,
%$$
%p = \colimite{\Gamma_p}{a \, \Diamond} \;\;=\;\; 
%\colimite{(x,\,X)}{\varepsilon_X} \;\;=\;\;
%\colimite{(x,\,X)}{a \,[-,\,X]} .
%$$
%
%
%\vspace{-4ex}  \cqd
%
%\end{lemma}
%
Recall the notion of morphism of topoi and their morphisms :

\begin{definition}[SGA4, IV, 3.1 and 3.2] \label{TOP}
Given any two topoi $\cc{X}$, $\cc{Y}$, a morphism 
\mbox{$\cc{X} \mr{u} \cc{Y}$} is a triple $u = (u^*, u_*, \varphi)$ with functors $\cc{Y} \pmrl{u^*}{u_*} \cc{X}$, $u^*$ exact, and an adjunction \mbox{$\varphi$: $u^* \dashv u_*$,} ($u^*$ is called the \emph{inverse image} and 
\mbox{$u_*$ the \emph{direct image}).}

Given two morphisms $u,\, v$, a morphism of morphims  is a natural transformation $u_* \mr{\beta} v_*$ (note that by means of the adjunctions,  $\beta$ uniquely determines a natural transformation 
$v_* \mr{\alpha} u_*$ and vice-versa). Composition is just the vertical composition of natural transformations. This defines a category $TOP[\cc{X},\, \cc{Y}]$.
\end{definition}
\begin{observation} 
Note that 
$TOP[\cc{X},\, \cc{Y}] \, =
                      \, ELA[\cc{Y},\, \cc{X}]^{op}$,
that is, objects and morphisms are both reversed
\footnote{Several authors have changed this Grothendieck's convention.}.                      
\end{observation}
%
% 
%$FlatCont[\cc{C},\, \cc{Z}] \pmrl{\widehat{(-)}}{h^\ast}  %AlgTop[\Eop{C},\, \cc{Z}]$ 
%
\section{Yoneda VI: Classifying site morphisms} \label{sec_yoneda6}

Note that Lemma \ref{yoneda5} in section \ref{sec_yoneda5} reads as follows: 

For any topos $\cc{Z}$, the functors 
\begin{equation} \label{prepreyoneda6}
Flat[\cc{C},\, \cc{Z}]^{op} \pmrl{\ell}{h^\ast} 
TOP[\cc{Z},\,\Eop{C}] \; (= ELA[\Eop{C},\, \cc{Z}]^{op})
\end{equation}
establish an equivalence of categories (we slightly abuse notation since $\Eop{C}$ is not a topos when $\cc{C}$ is not small). 

We will now generalize this to a general topos of the form $Sh(\cc{C})$ in place of the (illegitimate) presheaf category $\Eop{C}$.
%$\pml{a}{b}$   $\pmrl{a}{b}$  $\pmdu{a}{b}$

We say that a functor $\cc{C} \mr{p} \cc{Z}$ into a topos 
 is \emph{continuos}  if it sends covering families into strict  epimorphic families. That is, for each
 \mbox{$X_i \mr{} X \; \in \; Cov(X)$, $pX_i \mr{} pX$} is a strict epimorphic family in $\cc{Z}$. 
 %Note that through the natural isomorphism $p \mr{\eta} \widetilde{p}\,h$, it follows that 
%\mbox{$\widetilde{p}(h_{X_i}) \mr{} \widetilde{p}(h_{X})$} is also an strict epimorphic family.
\begin{proposition} \label{preyoneda6}
Given  a flat functor $p \in Flat[\cc{C},\, \cc{Z}]$, for any \mbox{$Z \in \cc{Z}$, $h_p(Z)$} is a sheaf if and only if $p$ is continuos.
\end{proposition}
\begin{proof}
Given a covering $X_i \mr{} X$, consider the diagrams:
$$
\xymatrix@C=6ex
   {
    pX_i \ar[r]_{} 
        \ar@/^4ex/[rr]^{\alpha_i}
  & pX \ar@{-->}[r]_{\alpha}
  & Z
   }    
\hspace{5ex}
\xymatrix@C=6ex
   {
    h_{X_i} \ar[r]_{} 
        \ar@/^4ex/[rr]^{\alpha_i}
  & h_X \ar@{-->}[r]_{\alpha}
  & h_p(Z)
   } 
%\xymatrix@C=6ex
%   {
%    pX_i \ar@/^4ex/[rrr]^{\alpha_i}
%       \ar[r]_{\approx\,\eta}
%  & \widetilde{p}(h_{X_i})  \ar[r]_{} 
%  & \widetilde{p}(h_X) \ar@{-->}[r]_{\alpha}
%  & Z
%   }
$$
%
%Given a compatible family $\alpha_i$ as in the first diagram, we should produce a unique $\alpha$ as indicated in the same diagram. In fact, by Yoneda's lemma \ref{yoneda1} we have a family 
%$\alpha_i \in h_p(Z)(X_i)$, which by definition of $h_p$ is a compatible family $\alpha_i$ as in the second diagram, which  yields a unique alpha in this diagram, which in turn correspond by adjunction to a unique $\alpha$ in the first diagram. 
%

 By Yoneda's lemma \ref{yoneda1} and definition \ref{def_hp} of $h_p$, compatible families $\alpha_i$ in both diagrams are in bijective correspondence, and the same olds for the arrows $\alpha$.

The proof follows by definition of continuos for $p$ (left diagram) and that of sheaf for $h_p(Z)$ (right diagram)             
\end{proof} 
Consider now the full dense subcategories:
$$ContFlat[\cc{C},\, \cc{Z}]^{op} \subset Flat[\cc{C},\, \cc{Z}]^{op}, \hspace{1ex} and \hspace{2ex} 
TOP[\cc{Z},\,Sh(\cc{C})] \subset TOP[\cc{Z},\,\Eop{C}]$$
of continuos functors  and
 of morphisms whose direct image lands in the full subcategory of sheaves $Sh(\cc{C}) \subset \Eop{C}$.
 \begin{corollary} \label{preVI}
 The functor 
 $\ell, \; \ell(p) = (\widetilde{p},\,h_p,\, \varphi)$ in diagram \ref{prepreyoneda6} restricts to the subcategories and establish an equivalence:
$$
ContFlat[\cc{C}, \, \cc{Z}]^{op} 
\mr{\ell} 
TOP[\cc{Z},\, Sh(\cc{C})]\; 
%(= ELA[Sh(\cc{C}),\, \cc{Z}]^{op})
$$
\end{corollary}
This is not yet the end, what we want is to show that 
$\ell$ is actually a quasi-inverse of the functor 
$\varepsilon^*$ of precomposition with the canonical functor 
$\cc{C} \mr{\varepsilon} Sh(\cc{C})$, 
$\varepsilon = a\,h$, Definition \ref{canonico}.
%We use the following trivial fact:
%\begin{fact}
%Let $\cc{A}, \; \cc{B}$ categories,  
%$\cc{A} \pmrl{l}{r} \cc{B}$ a pair of functors, $l$ essentially surjective, together with an isomorphism 
%$id \buildrel {\eta} \over \approx r\,l$. Then the pair 
%$l,\; r$ is an equivalence of categories (the remaining %isomorphism 
%\end{fact}  
We proceed as follows:

\vspace{1ex}

The proposition \ref{preyoneda6} shows that the adjunction $\widetilde{p} \dashv h_p$ actually determines an adjunction between the categories 
$\cc{Z}$ and the full subcategory $Sh(\cc{C})$,
 which we relabel $\varphi '$: $ \widetilde{p}\,' \dashv  h'_p$ to render things completely clear. We have the following diagrams, where $i$ is the inclusion of the full subcategory. The diagram on the right is the exterior diagram of the one on the left:
$$
\xymatrix@C=4ex@R=3ex
   { 
    \cc{C}  \ar[rr]^{h}
            \ar@/_2ex/[rrdd]_{p}
            \ar@/^5ex/[rrrr]_{\varepsilon}
   & {}     \ar@{}[dd]^(.35){\approx\,\eta}
   & \Eop{C}  \ar@{<-}@<1.5ex>[dd]^{h_p} 
              \ar@<-1.2ex>[dd]^{\!\widetilde{p}}
   &
   & {\hspace{1ex} 
     Sh(\cc{C})}  \ar@{<-}@<1.5ex>[ddll]^{h'_p} 
                  \ar@<-1.2ex>[ddll]^{\!\widetilde{p}\,'}
                  \ar@<1.5ex>[ll]_{i} 
                  \ar@{<-}@<-1.2ex>[ll]_{a}
   \\
   \\
   & {}
   & {\hspace{2ex} \cc{Z} \hspace{2ex}}
   }
\hspace{4ex}
\xymatrix@C=4ex@R=3ex
   { 
    {\cc{C}}  \ar[rr]^{\varepsilon}
          \ar@/_2ex/[rrdd]_{p}
   & {}   \ar@{}[dd]^(.35){\approx\,\eta'}
   & {Sh(\cc{C})}  \ar@{<-}@<1.5ex>[dd]^{h'_p} 
                   \ar@<-1.2ex>[dd]^{\!\widetilde{p}\,'}
   \\
   \\
   & {}
   & {\cc{Z}}
   }
$$
By definition $h_p = i \, h'_p$. Composing the adjoints we obtain $\widetilde{p} \approx \widetilde{p}\,' \, a$.  Then we define a natural isomorphism  
$p  \buildrel {\eta'} \over {\approx} \widetilde{p}\,'\,\varepsilon$ by the composite
\mbox{$p \buildrel {\eta} \over {\approx} \widetilde{p} \, h \approx  
\widetilde{p}\,'\,a\,h = \widetilde{p}\,'\,\varepsilon$.}
This proof done, we can safely suppress relabeling the adjunction $\widetilde{p} \dashv h_p$, 
%when considered between the categories $\cc{Z}$ and $Sh(\cc{C})$ %
then we write 
$p \buildrel {\eta} \over {\approx} \widetilde{p}\,\varepsilon$.

\vspace{1ex}

We prove now the following result (SGA4, IV, Proposition 4.9.4), that we call 
\emph{Yoneda VI}:

\begin{lemma}[Yoneda VI] \label{yoneda6}
For any topos $\cc{Z}$, 
the functor $\varepsilon^\ast$ of precomposition with $\varepsilon$,
establishes an equivalence of categories, with a quasi-inverse given by the cocontinuos extension
 $\widetilde{(-)}$:
$$
ContFlat[\cc{C}, \, \cc{Z}]^{op} 
\pmrl{\ell}{\varepsilon^\ast} 
TOP[\cc{Z},\, Sh(\cc{C})] 
$$
where 
$\ell(p) = (\widetilde{p},\,h_p,\, \varphi)$, see \ref{preVI}, and 
$\varepsilon^*(u^*,\, u_*,\, \varphi) = u^* \varepsilon$.
\end{lemma} 
\begin{proof}
In fact, given $p$, we already have 
$p \buildrel {\eta} \over {\approx} \widetilde{p}\,\varepsilon$, that is 
$id \buildrel {\eta} \over {\approx} \varepsilon^\ast\, \ell$. We apply $\ell$ and have 
$\ell \, \buildrel {\ell\,\eta} \over {\approx} \, \ell\,\varepsilon^\ast \ell$. But we know by \ref{preVI} that $\ell$ is, in particular, essentially surjective, thus it follows
$id \buildrel {} \over {\approx} 
\ell\,\varepsilon^\ast$.   
\end{proof}

From \ref{yoneda6} it follows a result that we will use in the next section:

\begin{proposition}[SGA4, IV, Remarque 1.3] \label{itself}
Let $\cc{Z}$ be a topos, then the canonical functor 
$\varepsilon: \cc{Z} \mr{\cong} Sh(\cc{Z})$ into the category of sheaves for the canonical pretopology is an equivalence.
\end{proposition}
\begin{proof}
Considering in \ref{yoneda6} the case $\cc{C} = \cc{Z}$ with the canonical pretopology:
$$
ContFlat[\cc{Z}, \, \cc{Z}]^{op} 
\pmrl{\ell}{\varepsilon^\ast} 
TOP[\cc{Z},\, Sh(\cc{Z})]\; 
(= ELA[Sh(\cc{Z}),\, \cc{Z}]^{op}),
$$
a straightforward computation (note that in this case since the topology is subcanonical $\varepsilon$ is full and faithful) shows that 
$\varepsilon$ and $\ell(id)$ establish an equivalence 
$
\cc{Z} 
\pmrl{\varepsilon}{\ell(id)} 
Sh(\cc{Z})
$. 
Note that $\ell(id) = \widetilde{id}$, is the continuos extension of the identity functor.
\end{proof}

\section{Yoneda VII: Flatness for sites with finite limits} \label{sec_yoneda7} 

\begin{lemma}[Yoneda VII] \label{yoneda7} 
If a category $\cc{C}$ has finite limits, a functor 
$\cc{C} \mr{p} \cc{Z}$ into a topos $\cc{Z}$ is flat provided it is exact.
\end{lemma} 
\begin{proof}
First we prove the case $\cc{Z} = \cc{E}ns$, and then we reduce the general case to this one.
Let $\cc{C} \mr{p} \cc{E}ns$, the following facts are easily obtained with basic category theory:

\vspace{1ex}

{\bf 1.} In $\cc{E}ns$ filtered colimits commute with finite limits (just check the canonical constructions). 

{\bf 2.} It follows immediately from {\bf 1.} that filtered colimits of exact functors are exact.

{\bf 3.} If $\cc{C}$ has finite limits and $p$ is exact, then the category 
$\Gamma_p$ is cofiltered (immediate, $\Gamma_p$ has finite colimits).

{\bf 4.} The continuos extension of a representable functor
$\cc{C} \mr{[X,\, -]} \cc{E}ns$ is 
$\Eop{C} \mr{ev_X} \cc{E}ns$, and therefore is exact.

{\bf 5.} Note that by \ref{yoneda3} 
$\;\widehat{(p)}\,$ as a functor in the variable $p$, preserves colimits.

\vspace{1ex}

Putting all this together, from the dual case of \ref{yoneda2} (Yoneda II) it follows that 
\emph{exact implies flat} for set-valued functors.

\vspace{1ex}

We pass now to the general case.
We have the following diagram:

$$ 
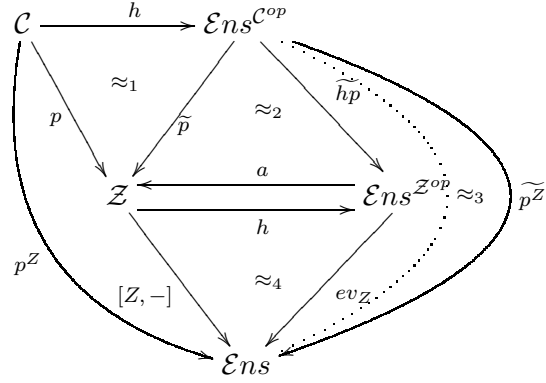
\xymatrix@C=4.5ex@R=4.5ex
   {
    {\cc{C}}  \ar[rr]^{h}  
              \ar[rdd]_{p} 
              \ar@/_8ex/[ddddrr]_{p^Z} 
              \ar@{}[drr]_{\approx_1}
   && {\Eop{C}} \ar[ldd]^{\hspace{-1ex}\widetilde{p}} 
                \ar[rdd]^{\widetilde{hp}} 
                \ar@/^21ex/[dddd]^{\widetilde{p^Z}}
                \ar@{.}@/^16ex/[dddd]^{\approx_3}
                \ar@{}[dd]^{\approx_2} 
   \\
   && {}
   \\ 
   & {\cc{Z}} \ar@<-1ex>[rr]_{h} 
              \ar[rdd]_{[Z,\, -]} 
   & {}       \ar@{}[dd]^{\approx_4}
   & {\Eop{Z}} \ar[ldd]^{ev_Z}
               \ar@<-1ex>@/0ex/[ll]_{a}
   \\
   \\
   && {\cc{E}ns}
   }
$$

\vspace{1ex}

%$$\xymatrix {X \ar@{.}@/^5ex/[rr]^o  &&  Y}$$
%
%$$\xymatrix@1{ A \ar[r]^*+<1ex>[o][p-]{2} & B }$$
%

The arrow $p^Z$ is defined to be the composite 
$[Z,\, -]\,p = [Z,\, p(-)]$, and the adjoint pair 
$\;a \dashv h\;$, $id \approx a \, h$, corresponds to the fact that there is an equivalence $\cc{Z} \simeq Sh(\cc{Z})$ for the canonical topology on $\cc{Z}$, Proposition  \ref{itself}. 

\vspace{1ex}

Clearly $p^Z$ is exact, so by the case $\cc{Z} = \cc{E}ns$ we have $\widetilde{p^Z}$ exact. Then by $\approx_3$ it follows that $\widetilde{hp}$ is exact. Since $a$ also is exact, $\approx_2$ shows that $\widetilde{p}$ is exact, and the proof finishes.  

\vspace{1ex}

We pass now to show the existence of the natural isomorphisms "$\approx_i$", $i = 1,\,2,\,3,\,4$ and 
label $\approx_5$ the exterior diagram.

We have $\approx_1$ and  $\approx_5$ by definition of the cocontinuos extension, and $\approx_4$ is actually a strict commutativity. 

$\widetilde{p} \approx_2 a \, \widetilde{hp}$: since the three functors are continuos, it is enough \mbox{(by Yoneda II)} to show 
$\widetilde{p}\,h \approx a \, \widetilde{hp} \, h$. In fact:   
$\widetilde{p}\,h \approx p \approx a\,h\,p 
\approx  a \, \widetilde{hp} \, h$.

$ev_Z \, \widetilde{hp}  \approx_3 \widetilde{p^Z}$: as before, it is enough to show 
$ev_Z \, \widetilde{hp} \, h  \approx \widetilde{p^Z} \, h$. In fact:
$ev_Z \, \widetilde{hp} \, h  \approx ev_Z \, h \, p 
\approx  [Z, \, -] \, p  = p^Z  \approx  \widetilde{p^Z} \, h$.
\end{proof}

%{\color{red}
%\section{Some conclusions}
%}
%{\color{cyan} 

\end{document}